\documentclass[12pt]{amsart}

\usepackage{amsmath,amssymb,amsthm,mathtools,mathdots,nicefrac,tikz,bbding,stmaryrd,float,multicol,caption,bbm}
\usepackage[mathscr]{euscript}
\usepackage[margin=1in, marginparwidth=1.75cm]{geometry}
\usepackage{color}
\usepackage{hyperref, float}
\hypersetup{colorlinks=true,citecolor=purple,linkcolor=purple}

\theoremstyle{definition}
\newtheorem{theorem}{Theorem}[section]

\newtheorem{lemma}[theorem]{Lemma}
\newtheorem{proposition}[theorem]{Proposition}
\newtheorem{definition}[theorem]{Definition}
\newtheorem{example}[theorem]{Example}

\newcommand{\Spec}{\text{Spec}}

\newcommand{\Min}{\text{Min}}

\newtheorem{manualtheoreminner}{Theorem}
\newenvironment{manualtheorem}[1]{%
  \renewcommand\themanualtheoreminner{#1}%
  \manualtheoreminner
}{\endmanualtheoreminner}

\newcommand{\interior}[1]{%
  {\kern0pt#1}^{\mathrm{o}}%
}

\usepackage[backrefs]{amsrefs}

\begin{document}

\title{Gluing Minimal Prime Ideals in Local Rings}
\author{C. Colbert}
\thanks{The first author was partially supported by a Lenfest grant from Washington and Lee University.}

\author{S. Loepp}

\maketitle

\begin{abstract}
Let $B$ be a reduced local (Noetherian) ring with maximal ideal $M$.  Suppose that $B$ contains the rationals, $B/M$ is uncountable and $|B| = |B/M|$.  
Let the minimal prime ideals of $B$ be partitioned into $m \geq 1$ subcollections $C_1, \ldots ,C_m$. 
We show that there is a reduced local ring $S \subseteq B$ with maximal ideal $S \cap M$ such that the completion of $S$ with respect to its maximal ideal is isomorphic to the completion of $B$ with respect to its maximal ideal and such that, if $P$ and $Q$ are prime ideals of $B$, then $P \cap S = Q \cap S$ if and only if $P$ and $Q$ are in $C_i$ for some $i = 1,2, \ldots ,m$.  
\end{abstract}

\section{Introduction}

Given a Noetherian ring $B$, it is often useful to find another Noetherian ring $S$ such that the prime ideals of $B$ and the prime ideals of $S$ are related in some specific desired way.  For example, if $P$ is a prime ideal of $B$, then, in many situations, passing to the localization $S= B_P$ is incredibly useful, in part because there is a one-to-one (inclusion preserving) correspondence between the prime ideals of $B$ contained in $P$ and the prime ideals of $B_P$.  Similarly, it is a standard technique in many settings to study the domain $B/P$ and, of course, the relationship between the prime ideals of $B$ and the prime ideals of $B/P$ is well understood.  In this paper, we consider the following question.  Let $B$ be a local (Noetherian) ring and suppose that $B$ has $n$ minimal prime ideals.  Let $m$ be an integer such that $1 \leq m \leq n$.  Is there a local subring $S$ of $B$ such that $S$ and $B$ have the same completion, and such that, when viewed as partially ordered sets (posets), $\Spec(B)$ and $\Spec(S)$ are the same except that $\Spec(B)$ has $n$ minimal elements and $\Spec(S)$ has $m$ minimal elements?  Informally, in this setting, we think of obtaining the partially ordered set $\Spec(S)$ by taking the partially ordered set $\Spec(B)$ and ``gluing" certain minimal nodes together while preserving everything else. We show that for a large class of local rings, such a subring does, in fact, exist.

We start with a reduced local ring $B$ with maximal ideal $M$ and we suppose that $B$ contains the rationals, $B/M$ is uncountable, and $|B| = |B/M|$.  
Our goal is to construct a local ring $S$ such that $S \subseteq B$, the completion of $S$ is isomorphic to the completion of $B$, and, $\Spec(S)$ and $\Spec(B)$ when viewed as partially ordered sets, are the same except for their minimal elements.  
Specifically, the main result of this paper is the following theorem.


\begin{manualtheorem}{\ref{biggluing}}
Let $B$ be a reduced local ring with maximal ideal $M$.  Suppose that $B$ contains the rationals, $B/M$ is uncountable and $|B| = |B/M|$.  
Suppose also that the set of minimal prime ideals of $B$ is partitioned into $m \geq 1$ subcollections $C_1, \ldots ,C_m$. Then there is a reduced local ring $S \subseteq B$ with maximal ideal $S \cap M$ such that 
\medskip
\begin{enumerate}
\item $S$ contains the rationals, \\
\item The completion of $S$ at its maximal ideal is isomorphic to the completion of $B$ at its maximal ideal, \\
\item $S/(S \cap M)$ is uncountable and $|S| = |S/(S \cap M)|$, \\
\item If $Q$ and $Q'$ are minimal prime ideals of $B$ then $Q \cap S = Q' \cap S$ if and only if there is an $i \in \{1,2, \ldots ,m\}$ with $Q \in C_i$ and $Q' \in C_i$, \\
\item The map $f:\Spec(B) \longrightarrow \Spec(S)$ given by $f(P) = S \cap P$ is onto and, if $P$ is a prime ideal of $B$ with positive height, then $f(P)B = P$.  In particular, if $P$ and $P'$ are prime ideals of $B$ with positive height, then $f(P)$ has positive height and $f(P) = f(P')$ implies that $P = P'$. \\
\end{enumerate}
\end{manualtheorem}

The properties of $f$ in Theorem \ref{biggluing} guarantee that it is an order-preserving onto map and, when $f$ is restricted to the prime ideals of $B$ with positive height, it is a poset isomorphism from the prime ideals of $B$ with positive height to the prime ideals of $S$ with positive height.  
In addition, $f$ maps all of the elements of a given $C_i$ to the same prime ideal of $S$.  Hence, one could think of $f$ as gluing all the prime ideals in each respective $C_i$ together while totally preserving everything else about the spectrum.
 Moreover, there is no restriction on how the minimal prime ideals of $B$ are glued; that is, one can choose the sets $C_1, \ldots ,C_m$ to be {\em any} partition of the set of minimal prime ideals of $B$.  We refer to Theorem \ref{biggluing} as The Gluing Theorem.
 
This type of gluing is done in \cite{SMALL2009} where the ring $B$ contains the rationals and is required to be complete.  We show that it is possible to do this type of gluing replacing the condition that $B$ is complete with the conditions that $B$ is reduced, $B/M$ is uncountable and $|B| = |B/M|$. In particular, whereas the gluing in \cite{SMALL2009} is done inside of a complete local ring, the gluing in this paper can be done inside a suitable localized polynomial ring which is not complete.  
To illustrate, we give two examples for which our main result applies, but Theorem 3.12 in \cite{SMALL2009} does not.


\begin{example}
Let $B = \mathbb{C}[x,y,z,w]_{(x,y,z,w)}/((x)\cap (y,z))$.  Note that $B$ satisfies the conditions of Theorem \ref{biggluing}, and it has two minimal prime ideals.  Using Theorem \ref{biggluing} with $m = 1$, we obtain a local ring $S$ contained in $B$ such that the completion of $S$ is $\mathbb{C}[[x,y,z,w]]/((x)\cap (y,z))$ and such that $S$ has the same prime ideal structure as $B$ except that $S$ has only one minimal prime ideal.  In particular, $S$ is a local domain that is not catenary (and hence, not excellent) and, since $S$ is a subring of $B$, all ideals of $S$ are generated by polynomials.
\end{example}

\begin{example}
Let $B = \mathbb{C}[x_1,x_2,x_3,x_4,x_5,x_6]_{(x_1,x_2,x_3,x_4,x_5,x_6)}/(x_1x_2x_3x_4x_5x_6)$.  Then $B$ satisfies the conditions of Theorem \ref{biggluing}, and it has six minimal prime ideals. Let $C_1 = \{(x_1),(x_2),(x_3)\}$, $C_2 = \{(x_4),(x_5)\}$, and $C_3 =\{(x_6)\}$.  Then there exists a local ring $S$ contained in $B$ such that the completion of $S$ is  $\mathbb{C}[[x_1,x_2,x_3,x_4,x_5,x_6]]/(x_1x_2x_3x_4x_5x_6)$ and such that $S$ has exactly three minimal prime ideals.  Moreover, the minimal prime ideals of $S$ are $(x_1) \cap S = (x_2) \cap S = (x_3) \cap S$, $(x_4) \cap S = (x_5) \cap S$, and $(x_6) \cap S$, and, if $P,Q \in \Spec(B)$ are not minimal prime ideals of $B$, then $S \cap P = S \cap Q$ if and only if $P = Q$.
\end{example}

Since the ring $S$ in Theorem \ref{biggluing} is a reduced local ring that contains the rationals, $S/(S \cap M)$ is uncountable and $|S| = |S/(S \cap M)|$, we can apply the theorem multiple times to obtain a descending chain of rings where the number of minimal prime ideals of the rings in the chain decreases.  We illustrate with an example.

\begin{example}
Let  $B_1 = \mathbb{C}[[x,y,z]]/(xyz)$.  By Theorem \ref{biggluing}, there is a reduced local ring $B_2$ contained in $B_1$ such that the maximal ideal of $B_2$ is $B_2 \cap (x,y,z)$, $B_2$ contains the rationals, $B_2/(B_2 \cap (x,y,z))$ is uncountable, $|B_2| = |B_2/(B_2 \cap (x,y,z))|$, the completion of $B_2$ is $B_1$, $B_2$ has exactly two minimal prime ideals $B_2 \cap (x) = B_2 \cap (y)$ and $B_2 \cap (z)$, and, if $P,Q \in \Spec(B_1)$ are not minimal prime ideals of $B_1$, then $B_2 \cap P = B_2 \cap Q$ if and only if $P = Q$.  We now apply Theorem \ref{biggluing} to $B_2$ to obtain a local ring $B_3$ contained in $B_2$ such that the completion of $B_3$ is $B_1$ and such that $B_3$ has only one minimal prime ideal, namely $B_3 \cap (x) = B_3 \cap (y) = B_3 \cap (z)$.  Moreover, if $P,Q \in \Spec(B_1)$ are not minimal prime ideals of $B_1$, then $B_3 \cap P = B_3 \cap Q$ if and only if $P = Q$. 
\end{example}


All rings in this article are commutative with unity.  If $R$ is a ring with exactly one maximal ideal and $R$ is not necessarily Noetherian, we say that $R$ is quasi-local.  If $R$ is both quasi-local and Noetherian, we say that $R$ is local.  We use $(R,M)$ to denote a local ring with maximal ideal $M$ and, if $(R,M)$ is a local ring, we use $\widehat{R}$ to denote the $M$-adic completion of $R$.  Finally, we use $\Min(B)$ to denote the set of minimal prime ideals of $B$.


\section{The Gluing Theorem}\label{GluingSection}

We are now ready to begin the proof of our main result, The Gluing Theorem.  Much of the work in our proof is inspired by techniques from \cite{SMALL2009}.  Throughout, $(B,M)$ will be a reduced local ring with $B/M$ uncountable.
To prove The Gluing Theorem, we start by gluing two minimal prime ideals of $B$ together, and then we induct to get the final result.  We begin our construction with the following useful definition.

\begin{definition}
Let $(B,M)$ be a reduced local ring with $B/M$ uncountable, and let $\Min(B) = \{Q_1, Q_2, \ldots ,Q_n\}$ with $n \geq 2$.  A quasi-local subring $(R, R \cap M)$ of $B$ is called a Minimal-Gluing subring of $B$, or an MG-subring of $B$, if $R$ is infinite, $|R| < |B/M|$, and $R \cap Q_1 = R \cap Q_2$.
\end{definition}

Note that, if $B$ in the above definition contains the rationals, then $\mathbb{Q}$ is an MG-subring of $B$.
To construct our final ring $S$ in The Gluing Theorem, we begin with $\mathbb{Q}$ and successively adjoin uncountably many elements while ensuring that our resulting rings remain MG-subrings of $B$.  To do this, we make use of the following result which can be thought of as a generalization of the prime avoidance lemma.

\begin{lemma}[\cite{heitmannUFD}, Lemma 3] \label{primeavoid}
Let $(B,M)$ be a local ring.  Let $C \subseteq \Spec(B)$, let $I$ be an ideal of $B$ such that $I \not\subseteq P$ for every $P \in C$, and let $D$ be a subset of $B$.  Suppose $|C \times D| < |B/M|$.  Then $I \not\subseteq \bigcup \{P + r \, | \, P \in C, r \in D\}.$
\end{lemma}

If $R$ is a subring of the ring $B$, and if $Q$ is a prime ideal of $B$, then the map $R/(Q \cap R) \longrightarrow B/Q$ is an injection, and so we can think of $R/(Q \cap R)$ as a subring of $B/Q$.  Suppose $(R, R \cap M)$ is an MG-subring of $B$. The next lemma gives sufficient conditions on an element $x \in B$ for $R[x]_{(R[x] \cap M)}$ to also be an MG-subring of $B$.

\begin{lemma}\label{preadjoining}
Let $(B,M)$ be a reduced local ring with $B/M$ uncountable, and let $\Min(B) = \{Q_1, Q_2, \ldots ,Q_n\}$ with $n \geq 2$.  Suppose $(R, R \cap M)$ is an MG-subring of $B$.  If $x \in B$ satisfies the condition that $x + Q_i \in B/Q_i$ is transcendental over $R/(Q_i \cap R)$ for $i \in \{1,2\}$, then $S = R[x]_{(R[x] \cap M)}$ is an MG-subring of $B$ with $|S| = |R|$.
\end{lemma}

\begin{proof}
Since $R$ is infinite, $|S| = |R|$, and we have $|S| < |B/M|$.  Now suppose $f \in R[x] \cap Q_1$.  Then $f = r_mx^m + \cdots + r_1x + r_0 \in Q_1$ where $r_j \in R$ for $0 \leq j \leq m$.  Since $x + Q_1$ is transcendental over $R/(R \cap Q_1)$, we have $r_j \in R \cap Q_1 = R \cap Q_2$.  Hence, $f \in Q_2$, and so $R[x] \cap Q_1 \subseteq R[x] \cap Q_2$.  Similarly, $R[x] \cap Q_2 \subseteq R[x] \cap Q_1$, and therefore $R[x] \cap Q_1 = R[x] \cap Q_2$.  It follows that $S \cap Q_1 = S \cap Q_2$, and so $S$ is an MG-subring of $B$.
\end{proof}

We now use Lemma \ref{preadjoining} to show that we can adjoin very specific elements to an MG-subring of $B$ to obtain a larger MG-subring of $B$.  Lemma \ref{adjoining} is very useful and will be employed several times.

\begin{lemma}\label{adjoining}
Let $(B,M)$ be a reduced local ring with $B/M$ uncountable, and let $\Min(B) = \{Q_1, Q_2, \ldots ,Q_n\}$ with $n \geq 2$.  Suppose $(R, R \cap M)$ is an MG-subring of $B$.  Let $b \in B$ and let $z \in B$ such that $z \not\in Q_1$ and $z \not\in Q_2$.  Let $J$ be an ideal of $B$ such that $J \not\subseteq Q_1$ and $J \not\subseteq Q_2$.  Then there is an element $w \in J$ such that $S = R[b + zw]_{(R[b + zw] \cap M)}$ is an MG-subring of $B$ with $|S| = |R|$.
\end{lemma}

\begin{proof}
Let $i \in \{1,2\}$, and suppose $b + tz + Q_i = b + t'z + Q_i$ with $t, t' \in B$.  Then $z(t - t') \in Q_i$ and since $z \not\in Q_i$, $t + Q_i = t' + Q_i$.  Therefore, $b + tz + Q_i = b + t'z + Q_i$  if and only if $t + Q_i = t' + Q_i$.  Let $D_i$ be a full set of coset representatives for the cosets $t + Q_i \in B/Q_i$ that make $b + zt + Q_i$ algebraic over $R/(R \cap Q_i)$.  Note that $|D_i| \leq |R|$.  Define $D = D_1 \cup D_2$ and $C = \{Q_1, Q_2\}$.  Then $|C \times D| \leq |R| < |B/M|$.  By Lemma \ref{primeavoid} using $I = J$, there is an element $w \in J$ such that $w \not\in \bigcup \{P + r \, | \, P \in C, r \in D\}.$ Then $b + zw + Q_i$ is transcendental over $R/(R \cap Q_i)$ for $i \in \{1,2\}$.  By Lemma \ref{preadjoining}, $S = R[b + zw]_{(R[b + zw] \cap M)}$ is an MG-subring of $B$ and $|S| = |R|$.  
\end{proof}

Recall that we want our final ring to have the same completion as $B$.  To achieve this, we use the following two propositions.

\begin{proposition}[\cite{heitmann}, Proposition 1]\label{prop:complete_machine}
If $(R,R\cap M)$ is a quasi-local subring of a complete local ring $(T,M)$, the map $R\to T/M^2$ is onto, and $IT\cap R=IR$ for every finitely generated ideal $I$ of $R$, then $R$ is Noetherian and the natural homomorphism $\widehat{R}\to T$ is an isomorphism.
\end{proposition}

The converse of Proposition \ref{prop:complete_machine} also holds (for a proof of this, see, for example, Proposition 2.4 in \cite{countableexcellent}).  That is, if $R$ is a local ring with completion $(T,M)$, then the map $R \longrightarrow T/M^2$ is onto and $IT \cap R = I$ for every finitely generated ideal $I$ of $R$.

\begin{proposition}\label{completionsame}
Let $(B,M)$ be a local ring and let $T = \widehat{B}$. Suppose $(S,S \cap M)$ is a quasi-local subring of $B$ such that the map $S \longrightarrow B/M^2$ is onto and $IB \cap S = I$ for every finitely generated ideal $I$ of $S$.  Then $S$ is Noetherian and $\widehat{S} = T$.  Moreover, if $B/M$ is uncountable and $|B| = |B/M|$ then $S/(S \cap M)$ is uncountable and $|S| = |S/(S \cap M)|$.
\end{proposition}

\begin{proof}
Since $T$ is the completion of $B$, the map $B \longrightarrow T/(MT)^2$ is onto.  Since $M^2 \subseteq (MT)^2 \cap B$, the map $B/M^2 \longrightarrow T/(MT)^2$ is well defined and onto. By hypothesis, the map $S \longrightarrow B/M^2$ is onto, and so the map $S \longrightarrow B/M^2 \longrightarrow T/(MT)^2$ is onto.  Let $I$ be a finitely generated ideal of $S$.  Then, since $T$ is the completion of $B$,  $IT \cap B = IB$.  It follows that $IT \cap S = (IT \cap B) \cap S = IB \cap S = I$.  By Proposition \ref{prop:complete_machine}, $S$ is Noetherian and $\widehat{S} = T$.  

Now suppose $B/M$ is uncountable and $|B| = |B/M|$.  Since $T$ is the completion of both $S$ and $B$, we have $S/(S \cap M) \cong B/M \cong T/MT$.  Hence, $S/(S \cap M)$ is uncountable and $|S/(S \cap M)| = |B/M|$.  Now $|S| \leq |B| = |B/M| = |S/(S \cap M)|$, and it follows that $|S| = |S/(S \cap M)|$.
\end{proof}

Because we will use Proposition \ref{completionsame} to show that our final ring has the same completion as $B$, we want our final ring to contain an element of every coset in $B/M^2$.  The next lemma shows that we can adjoin an element of a specific coset $b + M^2$ to an MG-subring of $B$ that will result in another MG-subring of $B$.  Later in this section (Theorem \ref{gluing}), we will adjoin elements from every coset in $B/M^2$.

\begin{lemma}\label{coset}
Let $(B,M)$ be a reduced local ring with $B/M$ uncountable, and let $\Min(B) = \{Q_1, Q_2, \ldots ,Q_n\}$ with $n \geq 2$.  Let $b \in B$ and suppose $(R, R \cap M)$ is an MG-subring of $B$. Then there exists an MG-subring $(S, S \cap M)$ of $B$ such that $R \subseteq S$, $|S| = |R|$, and $S$ contains an element of the coset $b + M^2$.
\end{lemma}

\begin{proof}
Since $n \geq 2$ and $B$ is local, $M \not\subseteq Q_i$  for $i \in \{1,2\}$, and so $M^2 \not\subseteq Q_i$ for $i \in \{1,2\}$.  Use Lemma \ref{adjoining} with $J = M^2$ and $z = 1$ to find $m \in M^2$ such that $S = R[b + m]_{(R[b + m] \cap M)}$ is an MG-subring of $B$ with $|S| = |R|$.  Note that $R \subseteq S$ and $S$ contains $b + m$, an element of the coset $b + M^2$.
\end{proof}

In light of Proposition \ref{completionsame}, we want to make sure that, if $S$ is our final ring, $IB \cap S = I$ for every finitely generated ideal $I$ of $S$.  Lemma \ref{firstclose} will help us do this.

\begin{lemma}\label{firstclose}
Let $(B,M)$ be a reduced local ring with $B/M$ uncountable, and let $\Min(B) = \{Q_1, Q_2, \ldots ,Q_n\}$ with $n \geq 2$.  Let $(R, R \cap M)$ be an MG-subring of $B$.  Then, for any finitely generated ideal $I$ of $R$ and for any $c \in IB \cap R$, there is an MG-subring $(S, S \cap M)$ of $B$ such that $R \subseteq S$, $|S| = |R|$, and $c \in IS$.  
\end{lemma}

\begin{proof}
Let $I =( y_1, \ldots ,y_m)$.  We induct on $m$.  If $m = 1$ then $I = aR$ for $a \in R$, and $c = au$ for some $u \in B$.  
If $a = 0$, then $S = R$ works.  So assume $a \neq 0$.

First suppose $a \not\in Q_1$.  Then $a \not\in Q_2$.  We claim that $S = R[u]_{(R[u] \cap M)}$ is the desired subring of $B$.  Suppose $f \in R[u] \cap Q_1$.  Then $f = r_mu^m + \cdots + r_1u + r_0$ where $r_i \in R$.  Hence, $a^mf = r_mc^m + \cdots + r_1ca^{m - 1} + r_0a^m \in R \cap Q_1 = R \cap Q_2$.  Since $a \not\in Q_2$, we have $f \in Q_2$, and so $R[u] \cap Q_1 \subseteq R[u] \cap Q_2$.  Similarly, $R[u] \cap Q_2 \subseteq R[u] \cap Q_1$, and so $R[u] \cap Q_1 = R[u] \cap Q_2$.  It follows that $S \cap Q_1 = S \cap Q_2$. Note that $R \subseteq S$, $|S| = |R|$, and $c \in IS$.

Now assume that $a \in Q_1$.  Then $a \in Q_2$.  Since $B$ is reduced, $B_{Q_1}$ is a field, and so $\mbox{ann}_B(a) \not\subseteq Q_1$.  Similarly, $\mbox{ann}_B(a) \not\subseteq Q_2$. Using Lemma \ref{adjoining} with $z = 1$, there exists $w \in \mbox{ann}_B(a)$ such that $S = R[u + w]_{(R[u + w] \cap M)}$ is an MG-subring of $B$ with $|S| = |R|$.  Now, $u + w \in S$ and $a(u + w) = au = c$, and so $c \in IS.$ This completes the base case of the induction.

Suppose that $m > 1$ and that the lemma holds for all ideals generated by fewer than $m$ elements.  We have $c = y_1b_1 + y_2b_2 + \cdots + y_mb_m$ for some $b_i \in B$. 

We first consider the case where $y_i \not\in Q_1$ for some $i$.  Without loss of generality, suppose $y_2 \not\in Q_1$.  Then $y_2 \not\in Q_2$.  Use Lemma \ref{adjoining} with $J = B$ to find $w \in B$ such that $S' = R[b_1 + y_2w]_{(R[b_1 + y_2w] \cap M)}$ is an MG-subring of $B$ with $|S'| = |R|$. Note that $$c = y_1b_1 + y_1y_2w - y_1y_2w + y_2b_2 + \cdots + y_mb_m = y_1(b_1 + y_2w) + y_2(b_2 - y_1w)+ \cdots + y_mb_m.$$    Now consider the ideal $(y_2, \ldots ,y_m)$ of $S'$ and let $c^* = c  -  y_1(b_1 + y_2w)$.  Then, $c^* \in (y_2, \ldots ,y_m)B \cap S'$.  By our induction assumption, there is an MG-subring $(S, S \cap M)$ of $B$ such that $S' \subseteq S$, $|S| = |S'|$, and $c^* \in (y_2, \ldots ,y_m)S$.  So we have $c^* = y_2s_2 + \cdots + y_ms_m$ for some $s_i \in S$.  Hence, $c = c^* + y_1(b_1 + y_2w) \in (y_1, \ldots ,y_m)S = IS$, and it follows that $S$ is the desired MG-subring of $B$.

We now consider the case where $y_i \in Q_1$ for all $i = 1,2 \ldots, m$.  Then $y_i \in Q_2$ for all $i = 1,2 \ldots, m$. As before, $\mbox{ann}_B(y_1) \not\subseteq Q_1$ and $\mbox{ann}_B(y_1) \not\subseteq Q_2$.  Use Lemma \ref{adjoining} with $J = \mbox{ann}_B(y_1)$ and $z = 1$ to find $w \in \mbox{ann}_B(y_1)$ such that $S' = R[b_1 + w]_{(R[b_1 + w] \cap M)}$ is an MG-subring of $B$ with $|S'| = |R|$.  Consider the ideal $(y_2, \ldots ,y_m)$ of $S'$ and let $c^* = c  -  y_1(b_1 + w)$. Then $c^* \in (y_2, \ldots ,y_m)B \cap S'$, so by our induction assumption there is an MG-subring $(S, S \cap M)$ of $B$ such that $S' \subseteq S$, $|S| = |S'|$, and $c^* \in (y_2, \ldots ,y_m)S$.  So we have $c^* = y_2s_2 + \cdots + y_ms_m$ for some $s_i \in S$.  Since $c = c^* + y_1(b_1 + w)$, we have $c \in (y_1, y_2, \ldots ,y_m)S = IS$, and it follows that $S$ is the desired MG-subring of $B$.
\end{proof}

To ensure that $B$ and our final ring $S$ have the same spectrum except at the minimal prime ideals, we guarantee that, if $J$ is an ideal of $B$ of positive height, then $S$ contains a generating set for $J$.  In Lemma \ref{generators}, we show that, for a particular ideal $J$ of $B$, we can start with an MG-subring, and adjoin appropriate elements so that the resulting ring is not only an MG-subring of $B$, but it also contains a generating set for $J$.

\begin{lemma}\label{generators}
Let $(B,M)$ be a reduced local ring with $B/M$ uncountable, and let $\Min(B) = \{Q_1, Q_2, \ldots ,Q_n\}$ with $n \geq 2$. Suppose $J$ is an ideal of $B$ with $J \not\subseteq Q_1$ and $J \not\subseteq Q_2$.  Let $(R,R \cap M)$ be an MG-subring of $B$.  Then there exists an MG-subring $(S, S \cap M)$ of $B$ such that $R \subseteq S$, $|R| = |S|$, and $S$ contains a generating set for $J$.
\end{lemma}

\begin{proof}
Let $J = (x_1, x_2, \ldots ,x_k)$.  By the prime avoidance theorem, there exists $z \in J$ such that $z \not\in Q_1$ and $z \not\in Q_2$.  Note that $M \not\subseteq Q_1$ and $M \not\subseteq Q_2$.
By Lemma \ref{adjoining}, there is an $m_1 \in M$ such that $R_1 = R[x_1 + m_1z]_{(R[x_1 + m_1z] \cap M)}$ is an MG-subring of $B$ and $|R_1| = |R|$.  Note that $(x_1 + m_1z, x_2, \ldots ,x_k) + MJ = J$, and so by Nakayama's Lemma, $(x_1 + m_1z, x_2, \ldots ,x_k) = J$.  Now repeat this procedure replacing $x_1$ with $x_2$, and $R$ with $R_1$ to find $m_2 \in M$ such that $R_2 = R_1[x_2 + m_2z]_{(R_1[x_2 + m_2z] \cap M)}$ is an MG-subring of $B$, $|R_2| = |R_1|$, and $J = (x_1 + m_1z, x_2 + m_2z, x_3, \ldots ,x_k)$.  Continue the procedure to find an MG-subring $R_k$ of $B$ such that $R \subseteq R_k$, $|R_k| = |R|$, $J = (x_1 + m_1z, x_2 + m_2z, \ldots ,x_k + m_{k}z)$, and $x_j + m_jz \in R_k$ for all $j = 1,2, \ldots ,k$.  Then $S = R_k$ is the desired MG-subring of $B$.
\end{proof}

Our strategy is to start with $\mathbb{Q}$ and successively adjoin uncountably many carefully chosen elements of $B$ to get our final ring.  In the process, we construct increasing chains of MG-subrings.  The next lemma ensures that the union of these increasing chains satisfy most properties of MG-subrings.

\begin{lemma}\label{unioning}
Let $(B,M)$ be a reduced local ring with $B/M$ uncountable, and let $\Min(B) = \{Q_1, Q_2, \ldots ,Q_n\}$ with $n \geq 2$.  Let $\Omega$ be a well-ordered index set and suppose that $(R_{\beta}, R_{\beta} \cap M)$ for $\beta \in \Omega$ is a family of MG-subrings of $B$ such that, if $\alpha, \mu \in \Omega$ with $\alpha < \mu$, then $R_{\alpha} \subseteq R_{\mu}$.  
Then $S = \cup_{\beta \in \Omega} R_{\beta}$ is an infinite subring of $B$ such that $S \cap Q_1 = S \cap Q_2$.  Furthermore, if there is some cardinal $\lambda < |B/M|$ such that $|R_{\beta}| \leq \lambda$ for all $\beta \in \Omega,$ and if $|\Omega| < |B/M|$, then $|S| \leq \mbox{max}\{\lambda, |\Omega|\}$, and $S$ is an MG-subring of $B$.
\end{lemma}

\begin{proof}
It is clear that $S$ is infinite and $S \cap Q_1 = S \cap Q_2$.  Now, suppose there is some cardinal $\lambda < |B/M|$ such that $|R_{\beta}| \leq \lambda$ for all $\beta \in \Omega,$ and $|\Omega| < |B/M|$.   Then $S \leq \lambda |\Omega| = \mbox{max}\{\lambda, |\Omega|\}$.  So, $|S| < |B/M|$, and it follows that $(S, S \cap M)$ is an MG-subring of $B$.
\end{proof}

The next two results show that we can construct a subring of $B$ that satisfies several of our desired properties simultaneously.  Before we state and prove the results, we state a technical definition.

\begin{definition}
Let $\Psi$ be a well-ordered set and let $\alpha \in \Psi$.  Define $\gamma (\alpha) = \sup\{\beta \in \Psi \, | \, \beta < \alpha\}$.
\end{definition}

\begin{lemma}\label{closingideals}
Let $(B,M)$ be a reduced local ring with $B/M$ uncountable, and let $\Min(B) = \{Q_1, Q_2, \ldots ,Q_n\}$ with $n \geq 2$.  Let $J$ be an ideal of $B$ with $J \not\subseteq Q_1$ and $J \not\subseteq Q_2$, and let $b \in B$.  Suppose $(R, R \cap M)$ is an MG-subring of $B$.  Then there exists an MG-subring $(S, S \cap M)$ of $B$ such that $R \subseteq S$, $|S| = |R|$, $b + M^2$ is in the image of the map $S \longrightarrow B/M^2$, $S$ contains a generating set for $J$, and $IB \cap S = I$ for every finitely generated ideal $I$ of $S$.
\end{lemma}

\begin{proof}\label{closing}
First use Lemma \ref{coset} to obtain an MG-subring $(R', R' \cap M)$ of $B$ such that $R \subseteq R'$, $|R'| = |R|$, and $R'$ contains an element of $b + M^2$.  Next, use Lemma \ref{generators} to get an MG-subring $(R'', R'' \cap M)$ of $B$ such that $R' \subseteq R''$, $|R''|= |R'|$, and $R''$ contains a generating set for $J$.
Define
$$\Psi = \{ (I,c) \, | \, I \mbox{ is a finitely generated ideal of } R'' \mbox{ and } c \in IB \cap R''\}.$$
Well-order $\Psi$ so that it has no maximal element, and let $0$ denote its first element.  Note that $|\Psi| \leq |R''| = |R|$.  We recursively define a family of MG-subrings $(R_{\mu}, R_{\mu} \cap M)$ of $B$ for each $\mu \in \Psi$ such that $|R_{\mu}| = |R|$ and, if $\alpha, \rho \in \Psi$ with $\alpha < \rho \leq \mu$, then $R_{\alpha} \subseteq R_{\rho}$.  Define $R_0 = R''$.  Now, for $\mu \in \Psi$ assume that $R_{\beta}$ has been defined for all $\beta < \mu$ such that $(R_{\beta}, R_{\beta} \cap M)$ is an MG-subring of $B$, $|R_{\beta}| = |R|$ and if $\alpha, \rho \leq \beta$ with $\alpha < \rho$, then $R_{\alpha} \subseteq R_{\rho}$.  Suppose $\gamma(\mu) < \mu$, and let $\gamma(\mu) = (I,c)$. Then define $(R_{\mu}, R_{\mu} \cap M)$ to be the MG-subring obtained from Lemma \ref{firstclose} such that $R_{\gamma(\mu)} \subseteq R_{\mu}$, $|R_{\gamma(\mu)}| = |R_{\mu}|$, and $c \in IR_{\mu}$.  On the other hand, if $\gamma(\mu) = \mu$, define $R_{\mu} = \cup_{\beta < \mu} R_{\beta}$.  In this case, by Lemma \ref{unioning}, $(R_{\mu}, R_{\mu} \cap M)$ is an MG-subring of $B$ with $|R_{\mu}| = |R|$.  In either case, we have that $(R_\mu, R_{\mu} \cap M)$ is an MG-subring of $B$, $|R_{\mu}| = |R|$, and if $\alpha, \rho \leq \mu$ with $\alpha < \rho$, then $R_{\alpha} \subseteq R_{\rho}$. 

Let $S_1 = \cup_{\mu \in \Psi}R_{\mu}$.  By Lemma \ref{unioning}, $(S_1, S_1 \cap M)$ is an MG-subring of $B$ and $|S_1| = |R|$.  Let $I$ be a finitely generated ideal of $R''$ and let $c \in IB\cap R''$.  
Then $(I,c) = \gamma (\mu)$ 
for some $\gamma(\mu) \in \Psi$ with $\gamma(\mu) < \mu$.
By construction, $c \in IR_{\mu} \subseteq IS_1$.  It follows that $IB \cap R'' \subseteq IS_1$ for every finitely generated ideal $I$ of $R''$.

Repeat this process with $R''$ replaced by $S_1$ to obtain an MG-subring $(S_2, S_2 \cap M)$ of $B$ with $S_1 \subseteq S_2$, $|S_2| = |S_1|$, and $IB \cap S_1 \subseteq IS_2$ for every finitely generated ideal $I$ of $S_1$.  Continue to obtain a chain of MG-subrings $R'' \subseteq S_1, \subseteq S_2, \subseteq \cdots$ with $S_i \subseteq S_{i + 1}$, $|S_{i + 1}| = |S_i|$ and $IB \cap S_i \subseteq IS_{i + 1}$ for every finitely generated ideal $I$ of $S_i$.

Let $S = \cup_{i = 1}^{\infty}S_i$.  By Lemma \ref{unioning}, $(S, S \cap M)$ is an MG-subring of $B$ with $|S| = |R|$.  Now suppose $I$ is a finitely generated ideal of $S$, and $c \in IB \cap S$.  Then $I = (s_1, \ldots ,s_k)$ for $s_i \in S$.  Choose $N$ such that $c, s_1, \ldots ,s_k \in S_N$.  Then $c \in IB \cap S_N \subseteq IS_{N + 1} \subseteq IS$.  It follows that $IB \cap S = I$, and so $S$ is the desired MG-subring of $B$.
\end{proof}

\begin{theorem}\label{gluing}
Let $(B,M)$ be a reduced local ring with $B/M$ uncountable and $|B| = |B/M|$.  Suppose $B$ contains the rationals and $\Min(B) = \{Q_1, Q_2, \ldots ,Q_n\}$ with $n \geq 2$.  Then there is a quasi-local ring $S \subseteq B$ with maximal ideal $S \cap M$ such that the map $S \longrightarrow B/M^2$ is onto, $IB \cap S = I$ for every finitely generated ideal $I$ of $S$, $Q_1 \cap S = Q_2 \cap S$, and, if $J$ is an ideal of $B$ with $J \not\subseteq Q_1$ and $J \not\subseteq Q_2$, then $S$ contains a generating set for $J$.
\end{theorem}

\begin{proof}
First note that if
$$ \Psi = \{J \mbox{ an ideal of } B \, | \, J \not\subseteq Q_1 \mbox{ and } J \not\subseteq Q_2 \}$$
then $|\Psi| \leq |B|$
and $|B/M^2| = |B/M| = |B|$.  Well-order $B/M^2$ using an index set $\Omega$ such that $0$ is the initial element of $\Omega$ and every element of $\Omega$ has fewer than $|\Omega|$ predecessors.  Let $b_{\alpha} + M^2$ be the element of $B/M^2$ corresponding to $\alpha \in \Omega$.  
Fix a surjective map $f$ from $\Omega$ to $\Psi$, and, for $\alpha \in \Omega$, define $J_{\alpha} = f(\alpha)$ .

We recursively define $R_{\alpha}$ for each $\alpha \in \Omega$.
First, define $R_0 = \mathbb{Q}$ and note that $\mathbb{Q}$ is an MG-subring of $B$.  Let $\alpha \in \Omega$ and assume $R_{\beta}$ has been defined for all $\beta < \alpha$ such that $(R_{\beta}, R_{\beta} \cap M)$ is an MG-subring of $B$ and $|R_{\beta}| \leq |\{\mu \in \Omega \, | \, \mu < \beta\}||R_0|$. If $\gamma(\alpha) < \alpha$, define $(R_{\alpha}, R_{\alpha} \cap M)$ to be the MG-subring of $B$ subring obtained from Lemma \ref{closingideals} such that $R_{\gamma(\alpha)} \subseteq R_{\alpha}$, $|R_{\gamma(\alpha)}| = |R_{\alpha}|$, $b_{\gamma(\alpha)} + M^2$ is in the image of the map $R_{\alpha} \longrightarrow B/M^2$, $R_{\alpha}$ contains a generating set for $J_{\gamma(\alpha)}$, and $IB \cap R_{\alpha} = IR_{\alpha}$ for every finitely generated ideal $I$ of $R_{\alpha}$.  Then $|R_{\alpha}| = |R_{\gamma(\alpha)}| \leq |\{\mu \in \Omega \, | \, \mu < \gamma(\alpha)\}||R_0| = |\{\mu \in \Omega \, | \, \mu < \alpha\}||R_0|$.
If $\gamma(\alpha) = \alpha$, define $R_{\alpha} = \cup_{\beta < \alpha} R_{\beta}$.  In this case, $|\{\mu \in \Omega \, | \, \mu < \alpha \}| < |\Omega| = |B/M|$, and $|R_{\alpha}| \leq |\{\mu \in \Omega \, | \, \mu < \alpha\}||R_0|$.  By Lemma \ref{unioning}, $(R_{\alpha}, R_{\alpha} \cap M)$ is an MG-subring of $B$.  

Define $S = \cup_{\alpha \in \Omega}R_{\alpha}$.  By construction, $S \cap M$ is the maximal ideal of $S$, the map $S \longrightarrow B/M^2$ is onto, $Q_1 \cap S = Q_2 \cap S$ and, if $J$ is an ideal of $B$ with $J \not\subseteq Q_1$ and $J \not\subseteq Q_2$, then $S$ contains a generating set for $J$.  Let $I = (s_1, \ldots s_m)$ be a finitely generated ideal of $S$ and let $c \in IB \cap S$.  Then, for some $\mu \in \Omega$
with $I'B \cap R_{\mu} = I'$, for every finitely generated ideal $I'$ of $R_{\mu}$
we have $c, s_1, \ldots ,s_m \in R_{\mu}$.  It follows that $c \in (s_1, \ldots ,s_m)B \cap R_{\mu} = (s_1, \ldots ,s_m)R_{\mu} \subseteq (s_1, \ldots, s_m)S = I$.  Hence, we have that $IB \cap S = I$ for every finitely generated ideal $I$ of $S$.
\end{proof}


Before we state and prove the Gluing Theorem, we make two observations.  First, suppose that $(B,M)$ is a local ring and $(S,S \cap M)$ is a subring of $B$ with the same completion as $B$.  Let $Q$ be a minimal prime ideal of $B$.  Then there is a minimal prime ideal $\widehat{Q}$ of $\widehat{B}$ such that $B \cap \widehat{Q} = Q$. Note that $S \cap \widehat{Q}$ is a minimal prime ideal of $S$. It follows that $S \cap Q = S \cap (B \cap \widehat{Q}) = S \cap \widehat{Q}$ is a minimal prime ideal of $S$.  Therefore, if $Q$ is a minimal prime ideal of $B$ then $S \cap Q$ is a minimal prime ideal of $S$.  

Second, suppose $S$ is a subring of the ring $B$ and $P$ is a prime ideal of $B$ with positive height satisfying $(S \cap P)B = P$.  Then $S \cap P$ is not a minimal prime ideal of $S$.  To see this, observe that since $P$ has positive height, it strictly contains a minimal prime ideal $Q$ of $B$.  If $S \cap P$ is a minimal prime ideal of $S$, then $S \cap P = S \cap Q$ and so $P = (S \cap P)B = (S \cap Q)B \subseteq Q$, a contradiction.

We are now ready to state and prove The Gluing Theorem.  In our proof, we use Theorem \ref{gluing} and induct on the number of minimal prime ideals of $B$. 

\begin{theorem}(The Gluing Theorem)\label{biggluing}
Let $(B,M)$ be a reduced local ring containing the rationals with $B/M$ uncountable and $|B| = |B/M|$.  
Suppose $\Min(B)$ is partitioned into $m \geq 1$ subcollections $C_1, \ldots ,C_m$. Then there is a reduced local ring $S \subseteq B$ with maximal ideal $S \cap M$ such that 
\medskip
\begin{enumerate}
\item $S$ contains the rationals, \\
\item $\widehat{S} = \widehat{B}$, \\
\item $S/(S \cap M)$ is uncountable and $|S| = |S/(S \cap M)|$, \\
\item For $Q, Q' \in \Min(B)$, $Q \cap S = Q' \cap S$ if and only if there is an $i \in \{1,2, \ldots ,m\}$ with $Q \in C_i$ and $Q' \in C_i$, \\
\item The map $f:\Spec(B) \longrightarrow \Spec(S)$ given by $f(P) = S \cap P$ is onto and, if $P$ is a prime ideal of $B$ with positive height, then $f(P)B = P$.  In particular, if $P$ and $P'$ are prime ideals of $B$ with positive height, then $f(P)$ has positive height and $f(P) = f(P')$ implies that $P = P'$. \\
\end{enumerate}
\end{theorem}

\begin{proof}
Let $|\Min(B)| = k$.  We proceed by induction on $k$.  If $k = 1$, then $S = B$ works.
So let $k > 1$ and assume that the result holds for rings with fewer than $k$ minimal prime ideals.  If $|C_i| = 1$ for every $i = 1,2, \ldots ,m$, then $S = B$ works, so assume that $|C_i| \geq 2$ for some $i$.  Without loss of generality, assume $|C_1| \geq 2$.  Let $Q_1, Q_2$ be distinct elements of $C_1$.  By Theorem \ref{gluing}, there is a quasi-local ring $S' \subseteq B$ with maximal ideal $S' \cap M$ such that the map $S' \longrightarrow B/M^2$ is onto, $IB \cap S' = I$ for every finitely generated ideal $I$ of $S'$, $Q_1 \cap S' = Q_2 \cap S'$, and, if $J$ is an ideal of $B$ with $J \not\subseteq Q_1$ and $J \not\subseteq Q_2$, then $S'$ contains a generating set for $J$.  Since $B$ is reduced and contains the rationals, $S'$ also satisfies these properties.  By Proposition \ref{completionsame}, $S'$ is Noetherian, $\widehat{S'} = \widehat{B}$, $S'/(S' \cap M)$ is uncountable and $|S'| = |S'/(S' \cap M)|$.  Consider the map $f': \Spec(B) \longrightarrow \Spec(S')$ given by $f'(P) = S' \cap P$.  Let $J \in \Spec(S')$.  Then there is a $P \in \Spec(\widehat{S'}) = \Spec(\widehat{B})$ such that $P \cap S' = J$.  Hence, $(J \cap B) \cap S' = P \cap S' = J$, and since $J \cap B \in \Spec{B}$, we have that $f'$ is onto.
If $P$ is a prime ideal of $B$ with $P \not\subseteq Q_1$ and $P \not\subseteq Q_2$, then  $S'$ contains a generating set for $P$ and so $f'(P)B = (S' \cap P)B = P$.  It follows that $S'$ has $k - 1$ minimal prime ideals. 


Consider the partition $C'_1, C'_2, \ldots, C'_m$ on $\Min(S')$ given by 
$C'_i = \{ Q \cap S' \, | \, Q \in C_i\}$.  Note that $|C'_1| = |C_1| - 1$ and $|C'_i| = |C_i|$ for $i = 2,3, \ldots ,m$.
By induction there is a reduced local ring $S \subseteq S' \subseteq B$ with maximal ideal $S \cap M$ such that $S$ contains the rationals, $\widehat{S} = \widehat{S'}= \widehat{B},$ $S/(S \cap M)$ is uncountable, $|S| = |S/(S \cap M)|$, for $Q, Q' \in \Min{S'}$, $Q \cap S = Q' \cap S$ if and only if there is an $i$ with $Q \in C'_i$ and $Q' \in C'_i$, the map $f'':\Spec(S') \longrightarrow \Spec(S)$ given by $f''(P) = S \cap P$ is onto and, if $P$ is a prime ideal of $S'$ with positive height, then $f''(P)S' = P$. 

Note that $f:\Spec(B) \longrightarrow \Spec(S)$ given by $f(P) = S \cap P$ satisfies $f = f'' \circ f'$.  Since $f'$ and $f''$ are onto, so is $f$.  Let $P$ be a prime ideal of $B$ with positive height.  Then $P \not\subseteq Q_1$ and $P \not\subseteq Q_2$ and so $f'(P)B = P$.  Now $f'(P)$ is a prime ideal of $S'$ of positive height, and so $f''(f'(P))S' = f'(P)$.  Therefore, $P = f'(P)B = (f''(f'(P))S')B = f(P)B$.  



Let $Q, Q' \in \Min(B)$.  Then $Q \cap S = Q' \cap S$ if and only if $(Q \cap S') \cap S = (Q' \cap S') \cap S$ if and only if there is an $i \in \{1,2, \ldots ,m\}$ such that $Q \cap S' \in C'_i$ and $Q' \cap S' \in C'_i$ if and only if $Q,Q' \in C_i$.  
\end{proof}



\end{document}